\documentclass[12pt,reqno]{amsart}

\usepackage{graphicx}
\usepackage{amsmath}
\usepackage{amssymb}
\usepackage{amscd}
\usepackage{hhline}
\usepackage{dcpic}
\usepackage{pictex}
\newtheorem{theorem}{Theorem}

\newtheorem{theoreme}{Theorem}

\newtheorem{cor}[theorem]{Corollary}
\newtheorem{lem}[theoreme]{Lemma}

\newtheorem{prop}[theorem]{Proposition}
\newcommand\bib[1]{\bibitem[#1]{#1}}

\newcommand\1{{\bf 1}}
\renewcommand\a{\alpha}
\renewcommand\b{\beta}

\newcommand\C{{\mathbb C}}

\renewcommand\d{\delta}

\newcommand\D{{\mathcal D}}

\newcommand\E{\mathcal{E}}

\newcommand\g{\mathfrak{g}}

\newcommand\h{\mathfrak{h}}

\renewcommand\l{\lambda}
\newcommand\La{\Lambda}
\newcommand\m{\frak m}

\newcommand\oo{\omega}
\newcommand\op[1]{\mathop{\rm #1}\nolimits}
\newcommand\ot{\otimes}
\newcommand\p{\partial}

\newcommand\R{{\mathbb R}}
\renewcommand\t{\tau}

\newcommand\vp{\varphi}

\newcommand\z{\sigma}
\newcommand\Z{{\mathbb Z}}
\newcommand{\comm}[1]{}

\begin{document}

 \title[The gap phenomenon for finite type systems]{The gap phenomenon in the dimension study of finite type systems}
 \author{Boris Kruglikov}
 \date{}
 \address{Department of Mathematics and Statistics, University of Troms\o, Troms\o\ 90-37, Norway. \quad
E-mail: boris.kruglikov@uit.no.}

 \vspace{-14.5pt}
 \begin{abstract}
In this paper several examples of gaps (lacunes) between dimensions of maximal and submaximal
symmetric models are considered, which include investigation of number of independent linear and
quadratic integrals of metrics and counting the symmetries of geometric structures and differential
equations. A general result clarifying this effect in the case, when the structure is associated
to a vector distribution, is proposed.
 \end{abstract}

 \maketitle

\section*{Introduction}

Finite type systems in differential geometry and differential equations are systems
whose solution spaces are finite-dimensional. They are encoded by integrable vector distributions
on manifolds. Solutions of the latter are integral surfaces, which by Frobenius theorem
can be found through integration of ODEs.

Several enumeration problems arise in this respect. The most important are counting of symmetries
and counting of the first integrals. Indeed, these object allow to solve the
differential equations effectively. The larger is their amount, the easier is the integration.

Notice that symmetries of differential equations can be of different types: internal, external,
higher etc. Lie-B\"acklund type theorems in some cases identify those, but in general their
spaces have different sizes. Existence of symmetries is a constraint on the system.

The models with the maximal amount of symmetries, conservations laws etc are of special
significance. Usually sub-maximal symmetric models play an important role in applications too.

The setup for the enumeration problem can be: micro-local
(the objects are defined locally in the space of jets), local (with respect to the manifold) or global.
While the first two cases are often the same (provided we impose the requirement of algebraic behavior
for the solutions, e.g. considering polynomial in momenta first integrals),
the latter two are always different: the global existence can have strong topological implications.

For instance, the number of Killing vector fields locally on 2-disk $D^2$ can be only 0, 1 or 3,
but globally on $T^2$ it is 0, 1 or 2. Similarly, the number of quadratic integrals locally
is 1, 2, 3, 4 or 6 \cite{D}, but globally it can be 2 or 6 on $S^2$ and 2 or 3 on $T^2$ \cite{Kol,Ki}.
In any case the maximal number means constant curvature.

We will restrict exclusively to local problems, so all differential equations and their solutions
are defined locally near a regular point.

The following phenomenon is known classically: the maximal and the next after it (sub-maximal)
amounts of the objects we count (symmetries, integrals etc) usually have a gap between them.

For instance the maximal amount of point transformations of 2nd order scalar ODEs is 8
(trivializable ODEs with $\op{sl}_3$ algebra of symmetries), elsewise there are no more than 3
symmetries according to A. Tresse classification \cite{Tr,K$_2$}.

In some cases this gap (we stick to usage of this term, resembling "spectral gap";
in Russian tradition the word "lacune" is adopted) can be explained by the structure theory
of simple Lie groups, but in some others the algebraic explanation is lacking.

In this paper we will demonstrate the gap phenomenon in a number of examples.
In the particular case, when the geometric structures are encoded by vector distributions,
we remark that the maximal symmetric model is unique and the maximal symmetry algebra is graded.
Then we give a tool to control the gap. Namely, using recent advances in the Tanaka theory, we
clarify the behavior of symmetric models and we relate this to algebraic and homological methods.

In general the reason for appearance of the gap and its size remains an interesting open problem.

The paper is organized as follows. In Section \ref{S1} we recall the general setup from geometric
theory of differential equations, giving the method to measure dimensional characteristics we discuss.
In Section \ref {S2a} we discuss the automorphisms of geometric structures. Most results here are classical.
In Section \ref{S2} we consider the problem of integrability of geodesic flows, and present old and
new knowledge on the number of second and higher order integrals.

Then in Section \ref{S3} we recall basics from the Tanaka theory together with a recent progress,
and formulate a method to find sub-maximal symmetric structures. This will be applied in Section \ref{S4}
to the geometry of rank 2 distributions in manifolds of dimension 5 and 6. The former is known
since E. Cartan \cite{C} and we give a new interpretation of his equations,
while the latter will contain new results and we sketch their higher-dimensional analogs.

We finish with conclusion, summarizing the results of the paper.

\medskip

\textsc{Acknowledgment.} This research was partially supported by the grant DAAD project Nr. 208068.
I would like to thank Vladimir Matveev and other organizers of the conference
"Finite Dimensional Integrable Systems in Geometry and Mathematical Physics 2011" in Jena for hospitality.
Some calculations in the last section were performed using
\textsc{Maple} package \textsl{DifferentialGeometry}.

\section{Solution space of differential equations}\label{S1}

The systems discussed in the Introduction can be re-cast into counting dimensions of
the solutions spaces for certain (systems of) differential equations. For our purposes we
can restrict to linear PDEs.

A differential equation $\E$ is represented geometrically as a submanifold in the space
of jets $J^k(M,N)$, where $M$ is the space of independent variables $x=(x^1,\dots,x^n)$ and
$N$ is the space of dependent variables $u=(u^1,\dots,u^m)$.
A choice of these coordinates on $J^0=M\times N$ yields canonical coordinates on the space $J^k$:
$(u^j_\z)$, where multi-index $\z=(i_1,\dots,i_n)$ has length $|\z|=\sum_{a=1}^n i_a\le k$.

A system $\E$ of differential equations on $u=u(x)$ can be considered as a submanifold in jets.
We will assume regularity, i.e. the projection maps $\pi_{k,i}:J^k\to J^i$
have constant ranks when restricted to $\E$ and its prolongations.
If $\E$ is given by differential equations $F_j=0$ of orders $k_j$, then the prolongation
to jet-order $l$ is given as the locus $\E_l\subset J^l(M,N)$ of equations $\D_\z F_j=0$, where
$\D_\z$ denotes the total derivative by multi-indices $\z$, and the maximal
amount of differentiations is $|\z|\le l-k_j$.

The symbols of $\E$ are defined as $g_i=\op{Ker}(d\pi_{i,i-1}:T\E_i\to T\E_{i-1})$.
If $x_i\in\E_i$ and $\pi_i(x_i)=x\in M$, we can identify $g_i(x_i)\subset S^iT^*_xM\ot T_xN$.

In general $\dim g_i$ as a function of $i$ grows in accordance with the Hilbert polynomial \cite{KL},
that can be calculated by the characteristic variety of $\E$.
System $\E$ is said to be of finite type if $g_i(x)=0$ for large $i$ and all $x$ (this holds iff
the complex characteristic variety is empty).

In other words, $\E$ is of finite type if for some jet-level $\pi_{l,l-1}:\E_l\to\E_{l-1}$
is a diffeomorphism, i.e. all the jets of $u$ of order $l$ are expressed through
the lower order jets. This uses regularity assumptions; if we relax the latter then the
conclusion holds on an open dense set (but $g_i(x)=0$ for large $i$ should hold uniformly in $x\in M$).
In this case the solution space can be parametrized by the Cauchy data.

 \begin{theorem}\label{Thm1}
Dimension of the solution space of a finite type system $\E$ is bounded as follows (the sum is finite)
 \begin{equation}\label{ineq}
\dim\op{Sol}(\E)\le\sup_M\sum_i\dim g_i(x).
 \end{equation}
 \end{theorem}
\noindent
For the regular case, when $\dim g_i(x)$ are constants, this is a folklore classical result;
for the general case we refer to \cite{K$_3$} (it is possible to use $\inf$ instead of $\sup$ in
the right hand side; the result also holds for analytic $\E$ which have finite type almost everywhere).

Equality in (\ref{ineq}) is attained only in the case $\E$ is formally integrable, i.e.
$\E$ is regular (this assumption we adopt from now on) and
$\pi_{i,i-1}:\E_i\to\E_{i-1}$ is a vector bundle. The latter means that all compatibility
conditions of $\E$ hold true.

Notice that if $\E$ is compatible, the symbols of the prolongations $g_l$ can be calculated
from the symbols of the equation $g_k$ by purely algebraic prolongations
$g_l=g_k^{(l-k)}=g_k\ot S^{l-k}T^*_xM\cap S^lT^*_xM\ot T_xN$, but in general (non-integrable)
case the algebraic prolongations can exceed the actual symbol spaces.

Consider a family of equations $\E_\alpha$ depending on a numeric or functional parameter $\alpha$.
The parametrization will be assumed algebraic-differential.
For instance $\E$ can be the Lie equation describing symmetries of a geometric
object and $\a$ can parametrize these objects (like Riemannian metrics or distributions
of prescribed dimensions/type). We will assume that the symbol spaces $g_i$ calculated
algebraically have the same dimensions independent of the parameter $\a$.

In this paper we we would like to understand what are the maximal dimensions $\dim\op{Sol}(\E_\a)$ and
what are the respective structures. This corresponds to the parameters $\a$ for which $\E_\a$
is formally (and so by finite type condition also locally) integrable.

It turns out that the possible dimensions of $\dim\op{Sol}(\E_\a)$ usually have a gap between
the maximal value and the next one. The basic reason behind this is the following.
Non-maximality means that some compatibility conditions fail, and so new equations
shall be added to the collection $\{F_j=0\}$. This reduces dimension of some symbol space $g_i$
and its prolongations, and so the right hand side of (\ref{ineq}) decreases.

We are going to demonstrate this effect with examples.

\section{Automorphisms of classical geometric structures}\label{S2a}

{\bf Riemannian structures\/} or pseudo-Riemannian structures on a manifold $M$
always have finite-dimensional isometry groups, which are solutions to
the Lie equation $\E=\{[\vp]^1_x:d_x\vp(g)=g\}\subset J^1(M,M)$.
The corresponding Lie algebra consists of vector fields $\xi$ satisfying $L_\xi(g)=0$.
The symbols of this equation are (we  identify $TM\simeq T^*M$ via $g$)
 $$
g_0=T_xM,\ g_1=\La^2T_x^*M,\ g_2=g_1^{(1)}=0,
 $$
whence $\dim g_0=n$, $\dim g_1=\frac{n(n-1)}2$.
Thus the system has finite type on the level of 2-jets, $\pi_{2,1}:\E_2\stackrel\sim\to\E_1$,
and the maximal dimension of the solution space is
 $$
\mathfrak{S}_\text{max}=\sum\dim g_i=\frac{n(n+1)}2.
 $$

This is realized precisely when $\E$ is compatible, which corresponds to the spacial forms
(constant sectional curvature), so that locally $M$ is round sphere $S^n$, Euclidean space $\R^n$ or
the Lobachevsky space $L^n$. The symmetry algebras are resp. $\mathfrak{so}(n+1)$,
$\mathfrak{so}(n)\ltimes\R^n$ or $\mathfrak{so}(1,n)$.

Let us discuss the submaximal dimension in the Riemannian case.
The stabilizer of a point $x_0$ is a subalgebra in $\mathfrak{so}(n)$. If it is full,
the translational part at $x_0$ should be either of dimension $n$ or zero.
In the first case we get the constant curvature, in the second - surface of revolution.
If the isotropy group is proper, then (as maximal) it is conjugated to $\mathfrak{so}(n-1)$,
provided $n\ne4$. Thus the metric again gives the surface of revolution $ds^2=dr^2+f(r)^2ds^2_{n-1}$.
Translation part is full only if the sectional curvature for $ds^2_{n-1}$ is constant and
to exclude the maximal case, we suppose the curvature is non-zero. The gap is thus $(n-1)$ for $n\ne2,4$,
i.e. the sub-maximal symmetry algebra of a Riemannian metric $g$ has dimension
 $$
\mathfrak{S}_\text{sub.max}=\frac{n(n-1)}2+1.
 $$
The sub-maximal homogeneous Riemannian spaces were described by
K.Yano, M.Obata and N.Kuiper for $n\ge4$, see ref. in \cite{Ko,YK},
and by E.Cartan in dimension 3 and I.Egorov in dimension 4 \cite{E$_1$}.

In dimension 2 the submaximal case has dimension 1 (this result belongs to
Darboux and K\"onings \cite{D}), and in dimension 4 the submaximal case has dimension 8
-- in both cases the gap is 2  \cite{E$_2$}.

In pseudo-Riemannian case the situation is more complicated. The maximal case still
corresponds to the manifold of constant sectional curvature $M^n(c)$, i.e. locally Minkovsky or Wolf space.
The submaximal case has the same dimension of the symmetry algebra in the case of Lorenzian signature
and $n\ne3,5$, but the amount of submaximal cases is greater.
According to \cite{P} for $n>5$, $n\ne7$ locally they are:
products $M^{n-1}(c)\times\R$, $\epsilon$-spaces and Egorov spaces.

\smallskip

{\bf Affine structure\/} is given by a choice of affine connection on $M$.
The Lie equation for symmetries $\E\subset J^2(M,M)$ at some point $x_2\in\E$
with $\pi_2(x_2)=x$ has symbols $g_0=T_xM$, $g_1=T_x^*M\ot T_xM$, $g_2=0$.

Thus the symmetry algebra of the affine connection has dimension at most
 $$
\mathfrak{S}_\text{max}=n+n^2.
 $$
The equality is attained iff $\E$ is compatible, and this happens precisely when $M$
is the usual affine space with flat connection \cite{Ei}.
The submaximal case for $n\ge4$ has dimension
 $$
\mathfrak{S}_\text{sub.max}=n^2.
 $$

\smallskip

{\bf Almost complex structure\/} is a field of operators $J$ on $TM$ satisfying $J^2=-\1$.
If the Nijenhuis tensor $N_J$ vanishes, the structure is integrable and the local symmetry
pseudogroup is infinite-dimensional. To guarantee finite-dimensionality one imposes a non-degeneracy
condition on the Nijenhuis tensor \cite{K$_4$}. If $n=\dim_\C M>2$, this says that the image
$N_J(\La^2TM)$ is the whole $TM$. For $n=2$ it is enough to claim that the image be a non-holonomic
rank 2 distribution.

Provided the dimension of the symmetry group is finite what is its maximal value? We know the
definite answer for $n=3$. According to \cite{K$_4$} the maximal symmetric non-degenerate
almost complex structure is the Calabi structure $J$ on $S^6$, the symmetry group being $G_2$.
We do not know the submaximal model for certain, but expect its symmetry algebra to be 9-dimensional
(then the gap will be equal to 5).

\smallskip

{\bf Conformal structure\/} is given by a choice of conformal class of (pseudo-)Riemannian metric.
The Lie equation $\E\subset J^1(M,M)$ for $n>2$ has symbols $g_0=T_xM$,
$g_1=\op{co}(n)=\R\oplus\mathfrak{so}(n)\simeq\R\oplus\La^2T^*_xM$, $g_2=T_x^*M$ and $g_3=0$.
Thus
 $$
\mathfrak{S}_\text{max}=\frac{(n+1)(n+2)}2.
 $$
Equality is attained in conformally flat case (loc. Euclidean space $\R^n$).
The same is true for the signature $(p,n-p)$\footnote{Added in proof:
after discussion with M.Dunajski and K.Melnik it became clear that the sub-maximal model in dimension
$n=4$ and neutral signature is given by the so-called pp-waves metric. This is the anti-self-dual twist free,
constant curvature solution of the Einstein equations \cite{Du}. Its symmetry algebra has dimension 9,
and so the gap is 6 in this case.}.

\smallskip

{\bf Projective structure\/} can be given as a projective class of torsion free connection on $M$.
The symbols of its Lie equation are $g_0=T_xM$, $g_1=T_x^*M\ot T_xM$, $g_2=T_x^*M$ and $g_3=0$.
Thus
 $$
\mathfrak{S}_\text{max}=n^2+2n.
 $$
Equality is attained in projectively flat case (loc. $\R\mathbb{P}^n$), see \cite{Ko}.
We do not know submaximal cases for conformal and projective structures.

\smallskip

Closely related to projective equivalences are geodesic equivalences between Riemannian
manifolds. The counting function for the latter -- the {\bf degree of geodesic mobility} --
has lacunarity similar to that in the theory of motion. See \cite{KM$^2$S} and the references
therein for complete description of gaps in the distribution of the degree of
geodesic mobility. The maximal and sub-maximal degrees of geodesic mobility for $n>2$ are
 $$
\mathfrak{S}_\text{max}=\frac{(n+1)(n+2)}2,\quad \mathfrak{S}_\text{sub.max}=\frac{(n-1)(n-2)}2+1.
 $$
Thus the gap is $3n-1$. The dimension of concircular fields, which are projective analogs of
the Killing fileds, have gap 3 as for them
 $$
\mathfrak{S}_\text{max}=n+1,\quad \mathfrak{S}_\text{sub.max}=n-2.
 $$

That degree of geodesic mobility exceeds 2 is closely related to integrability of geodesic flows
\cite{MT}, which is discussed in the next section (in particular for $n=2$, when the
degree of geodesic mobility equals $\dim Q_2$, see (\ref{Q}), we observe that the maximal
case has the same dimension as prescribed above, namely $\mathfrak{S}_\text{max}=6$,
while the sub-maximal case has dimension $\mathfrak{S}_\text{sub.max}=4$).

\section{Integrals of geodesic flows}\label{S2}

Let $(M,g)$ be a Riemannian or pseudo-Riemannian manifold. The geodesic flow of $g$
is the dynamical system on $T^*M$ (with the canonical symplectic form) given by
the Hamiltonian $H=\frac12\|p\|^2_g=\frac12g^{ij}(x)p_ip_j$. Its integrability via functions analytic in momenta
is equivalent, by Whittaker's theorem, to existence of polynomial in momenta integrals.
Since $H$ is quadratic we can restrict to integrals of pure degree $d$.

Consider the vector space ($\{,\}$ denotes the standard Poisson bracket)
 \begin{equation}\label{Q}
Q_d(g)=\{F\in C^\infty(T^*M):\{H,F\}=0\text{ and }\deg F=d\}.
 \end{equation}
We will be interested in its dimension depending on $g$. The whole algebra of polynomial integrals
$Q(g)=\oplus Q_d(g)$ is graded and Poisson, satisfying $\{Q_k,Q_l\}\subset Q_{k+l-1}$.

\medskip

{\bf Linear integrals}, also known as Killing vector fields, have the form $F=u^k(x)p_k$.
The coefficients of the quadratic expression $\{H,F\}=0$ give a system $\E$ of
$\frac{n(n+1)}2$ equations on $n$ unknown functions. The symbols of $\E$ have the following
dimensions (interpreted as the number of free jets, or as the amount of Cauchy data):
 $$
\dim g_0=n,\quad \dim g_1=n^2-\frac{n(n+1)}2=\frac{n(n-1)}2,\quad \dim g_2=0.
 $$
Its easy to see that this is the Lie equation for the Riemannian metric $g$ -- the case discussed
in \S\ref{S2a}.

Solution to this $\E$ gives the Lie algebra $\op{sym}(g)=Q_1(g)$. The other spaces $Q_d(g)$
are not Lie algebras, but are modules over $Q_1(g)$.

\medskip

{\bf Quadratic integrals}, also known as Killing-St\"ackel 2-tensors,
have the form $F=u^{kl}(x)p_kp_l$. The coefficients of the cubic
expression $\{H,F\}=0$ give a system $\E$ of $\frac{n(n+1)(n+2)}6$ equations on
$\dim g_0=\frac{n(n+1)}2$ unknowns. The higher symbols of $\E$ have the following dimensions:
 $$
\dim g_1=\frac{n(n^2-1)}3,\quad \dim g_2=\frac{n^2(n^2-1)}{12},
\quad \dim g_3=0.
 $$
Thus the system has finite type on the level of 3-jets, $\pi_{3,2}:\E_3\stackrel\sim\to\E_2$,
and the maximal dimension of the solution space is
 $$
\mathfrak{S}_\text{max}=\sum\dim g_i=\frac{(n+1)^2((n+1)^2-1)}{12}.
 $$
This is realized precisely when $\E$ is compatible, which corresponds to the spacial forms
(it works also for the pseudo-Riemannian case).

In fact, the maximal space of quadratic integrals $Q_2=S^2Q_1$.
This means that quadratic integrals are linear combinations of the products of the
linear ones. For instance, for the flat case $Q_1=\langle p_k,r_{ij}=x_ip_j-x_jp_i\rangle$,
so $Q_2$ is generated by $g_0=\langle p_ip_j\rangle$,
$g_1=\langle L_{ijk}=r_{ij}p_k\rangle$ (relations are $L_{ijk}=-L_{jik}$ and $L_{ijk}+L_{jki}+L_{kij}=0$)
and $g_2=\langle R_{ijkl}=r_{ij}r_{kl}\rangle$ (relations are the same as for algebraic curvature tensors).

 \begin{prop}
Submaximal symmetric Riemannian metrics have dimension of the solution space $Q_2$
given by
 $$
\mathfrak{S}_\text{sub.max}\ge\frac{n(n+1)}2+\frac{n^2(n^2-1)}{12}.
 $$
 \end{prop}

In fact we expect that the above is the submaximal number of quadratic integrals.
The claim follows from the following

 \begin{lem}\label{L1}
For the metric $g=x_1\sum_1^n dx_i^2$
 $$
\dim Q_2(g)=\frac{n(n+1)}2+\frac{n^2(n^2-1)}{12}.
 $$
 \end{lem}

 \begin{proof}
The Hamiltonian is $H=\frac1{x_1}\sum p_i^2$. Let $F=\sum b_{ij}p_ip_j$ be a quadratic integral.
The condition $\{H,F\}=0$ is equivalent to the PDE system on functions $b_{ij}=b_{ij}(x)$.
The coefficient of $p_1^3$ gives
 $$
\frac{\p b_{11}}{\p x_1}=-\frac{b_{11}}{x_1}\ \Rightarrow\ b_{11}=\frac{\b(x_2,\dots,x_n)}{x_1}.
 $$
The coefficients of $p_1^2p_i$ for $i>1$
give
 $$
2x_1\frac{\p b_{1i}}{\p x_1}+b_{1i}+\frac{\p\b}{\p x_i}=0\ \Rightarrow\
b_{1i}=\frac{\z_i(x_2,\dots,x_n)}{\sqrt{x_1}}-\frac{\p\b}{\p x_i}.
 $$
The coefficients of $p_1p_i^2$ and $p_1p_ip_j$ for $i,j>1$ give
 $$
b_{ij}=\t_{ij}(x_2,\dots,x_n)-2\sqrt{x_1}\Bigl(\frac{\p\z_i}{\p x_j}+\frac{\p\z_j}{\p x_i}\Bigr)
+2x_1\frac{\p^2\b}{\p x_i^2}+\frac{\b}{x_1}\delta_{ij}.
 $$
Substituting these expressions into the coefficients of $p_i^3$, $p_i^2p_j$ and $p_ip_jp_k$
and separating different powers of $x_1$ we obtain
that $\b$ is arbitrary inhomogeneous quadric, $\z=0$ and $\t$ correspond to the (maximal) space
of Killing 2-tensors in $(n-1)$ dimensions. Since the inhomogeneous quadrics in $(n-1)$ dimensions
are bijective with homogeneous quadrics in $n$ dimensions, we are done.
 \end{proof}

Notice that for the above structure $\dim Q_1(g)=\frac{n(n-1)}2$, $n>1$, so that for $n>2$ it is
not sub-maximal for the dimension of Killing fields.
Let us consider sub-maximal metrics for $\dim Q_1$; they turn out to be not submaximal
for the dimension of Killing 2-tensors.
 \begin{lem}\label{L2}
For $n>2$ and the metric $g=dx_1^2+ds^2_{n-1}$, where the latter is a metric
on $D^{n-1}(x_2,\dots,x_n)$ of constant sectional curvature $c\ne0$,
 $$
\dim Q_2(g)=1+\frac{n(n-1)}2+\frac{n^2(n^2-1)}{12}.
 $$
 \end{lem}

 \begin{proof}
We will consider only the case of positive curvature $c>0$.
Using stereographic coordinates on $M^{n-1}(c)$ the Hamiltonian on $T^*M$ writes
 $$
H=p_1^2+(R^2+x_2^2+\dots+x_n^2)^2(p_2^2+\dots+p_n^2).
 $$
The Killing fields (written as functions on $T^*M$) are
 $$
K_1=p_1,\ \ K_i^0=\Bigl(R^2-\sum_2^nx_j^2\Bigr)p_i+2x_i\sum_2^nx_jp_j,\ \
K_{ij}^1=x_ip_j-x_jp_i
 $$
(the indices $i,j$ run from 2 to $n$).

Write the quadratic integral $F=ap_1^2+p_1\sum_2^nb_ip_i+\sum_{i,j>1}c_{ij}p_ip_j$.
The coefficients of $p_1^3$ in the equation $\{H,F\}=0$ yields $\frac{\p a}{\p x_1}=0$.
Then the coefficients of $p_1^2p_i$ yields $b_i=b_i^0-(R^2+x_2^2+\dots+x_n^2)\,x_1\frac{\p a}{\p x_i}$,
where $b^0_i=b^0_i(x_2,\dots,x_n)$, $i>1$.

The coefficients of $p_ip_jp_k$ with $i,j,k>1$ yield the condition that $\sum_{i,j>1}c_{ij}p_ip_j$
is a Killing 2-tensor of the metric $ds_{n-1}^2$. Since the latter has constant sectional
curvature, we conclude that this contribution to $F$ is generated by pairwise products of Killing fields
in $(n-1)$ dimension:
 $$
\Bigl\langle\sum_{i,j>1}c_{ij}p_ip_j\Bigr\rangle=S^2\langle K^0_i,K^1_{ij}\rangle.
 $$

Substituting this and evaluating the coefficient of $p_1p_ip_j$ we obtain certain expressions linear
in $x_1$, for which free terms mean that $\sum_2^n b_i^0p_i$ is a Killing fields for $ds_{n-1}^2$.
Thus it is a linear combination of $K^0_i,K_{ij}^1$.

The coefficients of $x_1$ give a system on $a=a(x_2,\dots,x_n)$ of the second order in Frobenius form
(its second symbol vanish $g_2=0$). But it is inconsistent -- the
compatibility conditions give $a=\op{const}$. The claim follows.
 \end{proof}

Let us notice that in the case of Lemma \ref{L2} we have $Q_2(g)=S^2Q_1(g)$, but this equality
is wrong in the case of Lemma \ref{L1}.

\medskip

{\bf Cubic and higher degree integrals}. In any order $d$ and dimension $n$ the quantity $\dim Q_d(g)$
attains maximum precisely when $g$ is a spacial form. In addition $Q_d(g)=S^dQ_1(g)$ for such metrics $g$.

Almost nothing is known about submaximal cases, except for $n=2$. Consider the latter case.
It was shown in \cite{K$_1$} that for $d=3$ and $g$ being of non-constant curvature
$\dim Q_3(g)<7$ (for the maximal case this dimension is 10). Thus the gap is at least 4,
and we conjectured it is 6 (i.e. submaximal $\dim Q_3=4$ the same as submaximal $\dim Q_2$).

This conjecture was proved under additional
assumption that $g$ has a Killing field in \cite{MS}. The assumption looks natural -- one expects from
experience with degree $d=2$ that the metric of submaximal type possesses an additional linear integral.

\medskip

Let us notice that the success of the above approach is related to the fact that given $H$,
the system $\{H,F\}=0$ on coefficients of $F$ is overdetermined and of finite type. If we consider both
$H$ and $F$ as unknowns, the situation changes (we treat only the local case).

Denote by $\bar\E$ the system on coefficients of $H,F$. If one of $k=\deg H$ or $l=\deg F$
equals to 1, then $\bar\E$ is underdetermined (implying existence of functional families of
solutions). So we assume $l\ge k>1$ and use the identity $\deg\{H,F\}=k+l-1$.

In dimension $n=2$, due to existence of isothermal coordinates, the system is determined for all
$k,l$. In particular, for $k=2$ it has the structure of semi-Hamiltonian system of hydrodynamic type
\cite{BM}.

For $n=3$ the system $\bar\E$ is underdetermined for $k=2$, $l=2,3$ and determined for
$(k,l)=(2,4)$. For $n=4$ it is determined for $(k,l)=(2,2)$. For all other parameters and for $n\ge5$
$\bar\E$ is overdetermined (so dimensional restrictions on its solution space are expected).

\section{Tanaka theory: old and new}\label{S3}

Tanaka theory \cite{T} describes symmetries (automorphisms) of vector distributions $\Delta\subset TM$.
With every such distribution one associates the graded Tanaka algebra $\g$ defined as follows.

Define the weak derived flag $\{\Delta_i\}$ through the module of its sections by
$\Gamma(\Delta_{i+1})=[\Gamma(\Delta),\Gamma(\Delta_i)]$ with $\Delta_1=\Delta$.
We assume that our distribution is completely non-holonomic, i.e. $\Delta_k=TM$ for some $k$,
and that $\Delta$ is regular, i.e. ranks of $\Delta_i$ are constant along $M$.

The quotient sheaf $\m=\oplus_{i<0}g_i$, $g_i=\Delta_{-i}/\Delta_{-i-1}$ has a natural structure of a graded
nilpotent Lie algebra at any point $x\in M$. The bracket on $\m$ is induced by the commutator
of vector fields on $M$. Distribution $\Delta$ is called strongly regular if $\m(x)$, as a Lie algebra,
does not depend on the choice of $x\in M$.

The Tanaka prolongation $\g=\m\oplus\g_+=\g_{-k}\oplus\dots\g_{-1}\oplus\g_0\oplus\g_1\oplus\dots$ is
the maximal graded Lie algebra with negative graded part equal to $\m$. The non-negative
part $\g_+$ is a subalgebra, and it a-priori can be infinite-dimensional. The graded components can be
calculated algorithmically, see \cite{T,AK}. For instance $\g_0$ is the space of grading
preserving derivations of $\m$.

For strongly regular distributions N. Tanaka constructed an absolute parallelism on the prolongation
manifold, and this gives the way to define curvature and flatness \cite{T}. In particular, it shows the
standard model $(\op{Exp}(\m),\g_{-1})$ gives the maximal symmetric distribution
(among strongly regular distributions) with the symmetry algebra $\g$.

The condition of strong regularity was removed in \cite{K$_3$}. The following statement is a combination
of results in \cite{K$_3$,AK} (dimensional part is a paraphrase of Theorem \ref{Thm1};
in the first part of the theorem $\sup$ can be changed to $\inf$).

 \begin{theorem}
Let $\Delta$ be of finite type, i.e. for some $k$ we have: $\g_i(x)=0$, for all $i\ge k$, $x\in M$.
Then its symmetry algebra is majorized by the Tanaka algebra $\g$ in the sense
 $$
\dim\op{sym}(\Delta)\le\sup_M\sum_i\dim\g_i(x).
 $$
This inequality is an equality if and only if the distribution $\Delta$ is flat.
 \end{theorem}
Thus the maximal symmetry model is unique (isomorphic to the standard model),
which is not the case with maximal cases for other geometric structures.

The following statement re-phrases Theorem 3 from \cite{K$_3$}.

 \begin{theorem}\label{symTan}
Let $h$ be a Lie algebra symmetry of a distribution $\Delta$ on a manifold $M$,
and let $\mathfrak{g}$ be the Tanaka algebra of $\Delta$. Then there exists a
Lie algebra compatible filtration $F_i$ of $h$ such that the corresponding graded
Lie algebra is a subalgebra of the Tanaka algebra: $\mathfrak{h}\subset\mathfrak{g}$.
 \end{theorem}
This is to be compared to the following result of R. Zimmer \cite{Z}
(which uses the assumption $g_1(x)\subset\mathfrak{sl}(T_xM)$ instead of grading).
 \begin{theorem}
Let $M$ be a compact group, $G$ a real algebraic group acting on $M$ by
volume-preserving transformations
and $P$ be a $G$-structure. Then Lie algebra $\mathfrak{h}=\op{sym}(P)$ embeds into
$\g=\op{Lie}(G)$.
 \end{theorem}

Theorem \ref{symTan} can be applied as follows.

 \begin{cor}\label{corTan}
Any sub-maximal symmetric case is obtained via the following construction.

Consider a subalgebra $\mathfrak{h}\subset\mathfrak{g}$ and let $h$ be a filtered Lie
algebra that cannot be monomorphically mapped into $\mathfrak{g}$
(as a filtered algebra), but whose associated graded Lie algebra equals $\mathfrak{h}$.

Suppose that $h$ has a subalgebra $h_0$ of dimension $\dim\mathfrak{g}-\dim M$ and
that for a vector subspace $\Pi\subset h$ of dimension equal to $\op{rank}(\Delta)$ we have:
$[h_0,\Pi]\subset h_0+\Pi$, $h_0\cap\Pi=0$.

If the iterated brackets of $h_0+\Pi$ generate $h$, then the homogeneous
space $H/H_0$ (where $H,H_0$ are the corresponding Lie groups) possesses $H$-invariant
nonholonomic distribution corresponding to $\Pi$.
 \end{cor}

We call such an algebra $h$ a deformation of $\mathfrak{h}$.
Notice that if $h\subset\g$ then the symmetry algebra of the distribution on $H/H_0$
will be $\g$, i.e. the maximal instead of sub-maximal case.

First obstructions to deformations of Lie algebra structure are given by $H^2(\mathfrak{h},\mathfrak{h})$.
These cohomology groups can be however non-zero
even in rigid cases (this is what often happens when $\mathfrak{h}$ is not semi-simple).
We will see an example of this in the next section.

\section{Symmetries of rank 2 distributions}\label{S4}

Consider rank 2 distributions, at first in 5-dimensional space. According to Goursat
\cite{G} they are encoded as Monge underdetermined ODEs
 \begin{equation}\label{E0}
y'=F(x,y,z,z',z'').
 \end{equation}
The equation-manifold $\R^5(x,y,z,z_1,z_2)$ is equipped with the Pfaffian system
$\{dz-z_1dx,dz_1-z_2dx,dy-Fdx\}$ and the rank 2 distribution
$\langle\D_x=\p_x+z_1\p_z+z_2\p_{z_1}+F\p_y,\p_{z_2}\rangle$ is dual to it.

Internal symmetries of (\ref{E0}) are by definition the symmetries (automorphisms)
of this distribution.

Condition $F_{z_2z_2}\ne0$ guarantees that the symmetry algebra is finite-dimensional.
In this case E.\,Cartan \cite{C} showed that the dimension is bounded by 14,
and in the case of equality the maximal group is $G_2$. The maximal symmetric
model is unique up to equivalence and it is given by the celebrated Hilbert-Cartan equation
 \begin{equation}\label{HC}
y'=(z'')^2.
 \end{equation}

What about sub-maximal cases? They all are given by the following Monge equations
(in \cite{C} overdetermined involutive 2nd order PDE systems on the plane were
classified, but it is not difficult to establish the equivalence):
 \begin{equation}\label{submax25}
y'=(z'')^m.
 \end{equation}
Here $m=0,1$ corresponds to the case of Engel distribution in $\R^4$ which has the contact pseudogroup
of symmetries; if $2m-1\in\{\pm\frac13,\pm3\}$ the equation is equivalent to (\ref{HC}) and its symmetry is
the exceptional Lie group $G_2$; for other cases the group of symmetries $\mathfrak{g}$ has dimension 7
($m\in\R$ is the only invariant of both the group and the equation) and is given in generators as follows:
 \begin{multline*}
\hspace{-10pt}\op{sym}\{(\ref{submax25})\}=\langle
W_1=\p_x, W_2=\p_y, W_3=\p_z, W_4=x\p_x+y\p_y+2z\p_z+z_1\p_{z_1},\\
W_5=x\p_z+\p_{z_1}, W_6=my\p_y+z\p_z+z_1\p_{z_1}+z_2\p_{z_2}, W_7=\\
\hspace{10pt}z_2^{m-1}\p_x+(m-1)\!\int\!z_2^{2m-2}dz_2\cdot\p_y+(z_1z_2^{m-1}-\tfrac1my)\p_z+(1-\tfrac1m)z_2^m\p_{z_1}
\rangle.
 \end{multline*}
Abstractly the Lie algebra structure reads off from the structure equations on p.169 \cite{C}.
In the basis $\langle X_1,X_2,X_3,X_4,X_5,Y_1,Y_2\rangle$
dual to the Cartan coframe $\langle\oo_1,\oo_2,\oo_3,\oo_4,\oo_5,\varpi_1,\varpi_2\rangle$
the non-trivial commutators are
 \begin{gather*}
[X_1,Y_1]=2X_1,\ [X_2,Y_2]=X_1,\ [X_3,X_4]=X_1,\\
[X_2,Y_1]=X_2,\ [X_3,Y_1]=X_3,\  [X_4,Y_1]=X_4,\  [Y_2,Y_1]=Y_2,\\
[X_2,X_5]=(IX_3+Y_2),\ [X_3,X_5]=(X_2+\tfrac43IX_4),\\
[X_4,X_5]=(X_3-IY_2),\ [X_5,Y_2]=X_4.
 \end{gather*}
The parameter $I$ is a semi-invariant and $I^2$ is a bona-fide invariant
(Cartan \cite{C} writes that $I$ is obtained from the forth root, so some power of it is an 
invariant; we shall see that this power is 2).

The above description of $\mathfrak{g}$ is however not convenient, since
the restriction of $\op{ad}(X_5)$ to the derived algebra
$\mathfrak{g}_2=\langle X_1,X_2,X_3,X_4,Y_2\rangle$
is not normalized ($\op{ad}(Y_1)$ acts like the grading element).
This operator is semi-simple and diagonalizing it
(changing the basis in $\langle X_2,X_3,X_4,Y_2\rangle$) is
equivalent to passing to the basis $\{W_i\}_{i=1}^7$.

In this new basis the description of $\mathfrak{g}$ is the following.
The derived series is $\mathfrak{g}_2=\langle W_1,W_2,W_3,W_5,W_7\rangle$,
$\mathfrak{g}_3=\langle W_3\rangle$, and $\mathfrak{g}_2$ is the Heisenberg algebra
$\mathfrak{h}=\mathfrak{h}_{-1}\oplus\mathfrak{h}_{-2}$ (now indices denote the grading)
given by the symplectic form on $\mathfrak{h}_{-1}=\langle W_1,W_2,W_5,W_7\rangle$
with values in $\mathfrak{h}_{-2}=\langle W_3\rangle$. In other words, the non-trivial
brackets are
 $$
[W_1,W_5]=W_3,\ [W_2,W_7]=\tfrac{-1}mW_3
 $$
(we keep the normalizing factor).

This is extended to $\tilde{\mathfrak{h}}=\mathfrak{h}\oplus\R\cdot W_4$ by the grading
element $\op{ad}(W_4)|_{\mathfrak{h}_k}=k\cdot\op{id}$.
Finally $\mathfrak{g}$ is obtained from $\tilde{\mathfrak{h}}$ by right extension
via the element in $H^1(\tilde{\mathfrak{h}},\tilde{\mathfrak{h}})$ given as follows
($m\ne1/2$):
 $$
-\op{ad}(W_6)=mW_2\ot\theta_2+W_3\ot\theta_3+W_5\ot\theta_5+(1-m)W_7\ot\theta_7
 $$
($\theta_i$ is the coframe dual to $W_i$).

Subtracting the trace we get the operator $A=\op{ad}(W_6-\frac12W_4)$ on $\mathfrak{h}_{-1}$
(it vanishes on $\mathfrak{h}_{-2}$), whose conformal class is an invariant of the Lie algebra
$\mathfrak{g}$.

From the spectrum $\op{Sp}(A)=\{\pm\frac12,\pm(\frac12-m)\}$ we obtain the absolute
invariant $\lambda=\op{Tr}(A^4)/\op{Tr}(A^2)^2$. Passing to $2(1-2\lambda)$ we get the
invariant
 $$
J=\frac{(1-2m)^2}{(1-2m+2m^2)^2}.
 $$
Calculating the invariant $J$ in the Cartan basis (using $\op{ad}(X_5)$) we get
$J=\frac9{25}(1+I^{-2})$. Now we can relate the parameters:
 $$
I=\pm\frac{i(1-2m+2m^2)}{2\sqrt{(m+1)(m-1/3)(m-2/3)(m-2)}}.
 $$
We see that poles of $I$ correspond to maximal finite-dimensional symmetry algebra $G_2$,
while linearizable cases correspond to $I=\pm\frac34$.

For the exceptional case $m=1/2$ the spectrum is multiple (this happens also for $m=0,1$),
and this is the only case, when $A$ is not semi-simple:
 $$
-\op{ad}(W_6-\tfrac12W_4)=\tfrac12W_5\ot\theta_5-\tfrac12W_1\ot\theta_1+\tfrac12W_2\ot\theta_7.
 $$

It is more convenient to describe the equivalence classes in terms of the parameter $k=2m-1$.
On the real line $\R^1(k)$ we have the action of the group $\Z_2\oplus\Z_2$ with generators
$k\mapsto-k$, $k\mapsto k^{-1}$ (the latter is not defined at 0). Its orbits correspond 
to the equivalent Monge equations (\ref{submax25}). In particular, the cases of symmetry dimensions 
$\infty$ and $14$ corresponds to the orbits of $k=1$ and $k=3$, and the exceptional case of 
Jordan block corresponds to $k=0$. 
 The orbit space is the union of two $k$-intervals $[0,1/3]\cup[1/3,1]$.
The invariant $I$ expresses via $k$ so:
 $$
I^2=\frac{(k^2+1)^2}{(k^2-9)(1/9-k^2)}. 
 $$
 
\medskip

Consider the next case of rank 2 distributions in $\R^6$, which are not
reducible to the previous case of $(2,5)$-distributions. Then \cite{DZ,AK}
its symmetry algebra is at most 11-dimensional. It is enough to restrict to the following
Monge equations in this case (in dimension 6 not all rank 2 distributions are realized
as this underdetermined ODE, the other two cases correspond to hyperbolic structures involving
only second derivatives and the elliptic structures that can be realized only as Monge systems
\cite{AK} but their symmetry algebras are at most 8-dimensional)
 \begin{equation*}
y'=F(x,y,z,z',z'',z''').
 \end{equation*}
Non-degeneracy condition $F_{z_3z_3}\ne0$ is sufficient for finite-dimensionality.
Here the maximal symmetric model is unique and is given by
 \begin{equation}\label{max13}
y'=(z''')^2.
 \end{equation}
We refer to \cite{AK} for the description of the symmetry algebra (both as internal and
as external higher symmetries).

The submaximal case is given by the family ($\epsilon\in\C$)
 \begin{equation}\label{submax26}
y'=(z''')^2+\epsilon^2(z'')^2.
 \end{equation}
The symmetry algebra is 9-dimensional and is represented as follows:
 \begin{multline*}
\op{sym}\{(\ref{submax26})\}=\langle W_1=\p_x,W_2=\p_y,W_3=\p_z,\\
W_4=2y\p_y+z\p_z+z_1\p_{z_1}+z_2\p_{z_2}+z_3\p_{z_3},W_5=x\p_z+\p_{z_1}, \\
W_6=2\epsilon^2z_1\p_y+\tfrac12x^2\p_z+x\p_{z_1}+\p_{z_2}, \\ W_7=2(z_2+\epsilon^2(xz_1-z))\p_y+\tfrac16x^3\p_z+\tfrac12x^2\p_{z_1}+x\p_{z_2}+\p_{z_3}, \\
W_8=e^{-\epsilon x}(2\epsilon^3z_2\p_y-\p_z+\epsilon\p_{z_1}-\epsilon^2\p_{z_2}+\epsilon^3\p_{z_3}),\\
W_9=e^{\epsilon x}(2\epsilon^3z_2\p_y+\p_z+\epsilon\p_{z_1}+\epsilon^2\p_{z_2}+\epsilon^3\p_{z_3})\rangle.
 \end{multline*}
Its Lie algebra structure is the following: the derived series of $\mathfrak{g}$ is
$\mathfrak{g}_2=\langle W_2,W_3,W_5,W_6,W_7,W_8,W_9\rangle$, $\mathfrak{g}_3=\langle W_2\rangle$,
where $\mathfrak{g}_2$ is the Heisenberg algebra again
$\mathfrak{h}=\mathfrak{h}_1\oplus\mathfrak{h}_2$ (the indices denote the grading)
given by the symplectic form on $\mathfrak{h}_1=\langle W_3,W_5,W_6,W_7,W_8,W_9\rangle$
with values in $\mathfrak{h}_2=\langle W_2\rangle$:
 \begin{gather*}
[W_3,W_7]=-2\epsilon^2W_2,\ [W_5,W_6]=2\epsilon^2W_2,\\ [W_6,W_7]=2W_2,\ [W_8,W_9]=-4\epsilon^5W_2
 \end{gather*}
(normal form is achieved by the shift of
$(W_6,W_7)$ by $\frac1{2\epsilon^2}(W_3,W_5)$). Then the 2-dimensional right extension from
$\mathfrak{h}$ to $\mathfrak{g}$ is achieved via grading element $W_4$,
$\op{ad}(W_4)|_{\mathfrak{h}_k}=k\cdot\op{id}$,
and the derivation
 $$
\op{W_1}=W_3\ot\theta_5+W_5\ot\theta_6+W_6\ot\theta_7-\epsilon W_8\ot\theta_8+\epsilon W_9\ot\theta_9.
 $$
Let us now demonstrate the sub-maximal property.

 \begin{theorem}
There exists no rank 2 distribution in $\R^6$ with symbol of general position
(i.e. growth vector being $(2,3,5,6)$) that possesses a 10-dimensional symmetry algebra.
 \end{theorem}

To demonstrate this we will use Corollary \ref{corTan}.

 \begin{proof}
The maximal symmetry algebra, corresponding to (\ref{max13}), was calculated in \cite{AK}:
the graded Lie structure is
 \begin{multline*}
\g=\g_{-4}\oplus\g_{-3}\oplus\g_{-2}\oplus\g_{-1}\oplus\g_0\oplus\g_1\\
=\langle Z_0\rangle\oplus\langle Y_0,Z_1\rangle\oplus\langle Z_2\rangle\oplus
\langle S_0,Z_3\rangle\oplus\langle S_1,R,Z_4\rangle\oplus\langle S_2,Z_5\rangle
 \end{multline*}
and the commutators are (we list only non-trivial ones)
 \begin{gather*}
[S_0,S_1]=S_0,\ [S_0,S_2]=2S_1,\ [S_1,S_2]=S_2,\\
[S_0,Z_i]=Z_{i-1},\ [S_1,Z_i]=(i-\tfrac52)Z_i,\ [S_2,Z_i]=(i+1)(i-5)Z_{i+1},\\
[Z_0,Z_5]=2Y_0,\ [Z_1,Z_4]=-2Y_0,\ [Z_2,Z_3]=2Y_0\\
[Y_0,R]=Y_0,\ [Z_i,R]=\tfrac12Z_i.
 \end{gather*}
Its Levi decomposition is
 $$
\g=\op{sl}_2+\mathfrak{m}^8=\langle S_i\rangle\oplus\langle R,Y_0,Z_j\rangle.
 $$
It follows that any 10-dimensional graded subalgebra shall contain the 8-dimensional space
$\langle S_1,Y_0,Z_j\rangle$, to which a 2-dimensional subspace in $\langle S_0,S_2,R\rangle$ shall be added.

This can be reduced to removing from $\g$ either $R$ or $S_2$ (in fact, internally any subalgebra of
$\op{sl}_2$ is conjugated to $\langle S_0,S_1\rangle$, but in our case the semi-simple part enters with
the graded representation, and so there are more cases; however e.g. removal of $S_0$ is much simpler
and we restrict to the most complicated cases). This is a 10-dimensional subalgebra $\h\subset\g$,
but we need to check if it has a deformation in the spirit of Corollary \ref{corTan}.

The first choice $\h=\langle S_i,Y_0,Z_j\rangle$ is easy: The $\op{sl}_2$ subalgebra is rigid
upon deformation $\h\rightsquigarrow h$, as well as its module $\mathfrak{m}^7=\langle Y_0,Z_j\rangle$.
The Lie algebra structure on $h=\op{sl}_2+\mathfrak{m}^7$ is given by a choice of
$\op{sl}_2$-equivariant morphism $\La^2\mathfrak{m}^7\to h$.

Denoting by $V$ the tautological (2D) representation and by
$\varepsilon$ the trivial one, we have $\mathfrak{m}^7=S^5V\oplus\varepsilon$ and
consequently $\La^2\mathfrak{m}^7=S^8V\oplus S^5V\oplus S^4V\oplus\varepsilon$.
Notice that the adjoint representation, equivalent to $S^2V$ is not among the summands.
Thus by Schur's lemma the only morphism is the projection to $S^5V\oplus\varepsilon$
(given by two coefficients).

Therefore the only possible Lie algebra structure is given by the commutators
$[Z_i,Z_j]=\l_{ij}Y_0$ and $[Z_i,Y_0]=cZ_i$ (the brackets with $\op{sl}_2$ are the same as in $\g$).
Acting on the first equation by $\op{ad}(S_1)$ we obtain $\l_{ij}=0$ unless $i+j=5$,
and acting on it by $\op{ad}(S_0)$ we get $\l_{0,5}=-\l_{1,4}=\l_{2,3}$ and this can be normalized to
2 (it is non-zero by the deformation constraints). Now the Jacobi identity for the triple $(Z_i,Z_{5-i},Z_k)$
yields $c=0$ and we recover the structure of $\h$. So no non-trivial deformation is possible

The second choice  $\h=\langle S_0,S_1,Y_0,Z_j,R\rangle$ is more complicated because the semisimple
part is gone. Let us start by calculating the usual deformation.
The operator $\op{ad}(S_1)$ has spectrum $\{-1,0,0,0,\pm\frac52,\pm\frac32,\pm\frac12\}$.
The simple eigenvalues will be perturbed and we can normalize $-1$ to be fixed.
This restores the solvable subalgebra $\langle S_0,S_1\rangle$.

Next the Jacobi identity yields $[S_1,[S_0,Z_i]]=(\l_i-1)[S_0,Z_i]$ provided $[S_1,Z_i]=\l_iZ_i$.
This implies $[S_0,Z_i]=Z_{i-1}$ and $[S_1,Z_i]=(i-\frac52+\l)Z_i$ for some $\l$.
We define $Y_0=\frac12[Z_0,Z_5]$ and then calculate if $\l\ne0$ but small: $[S_0,Y_0]=0$,
$[S_1,Y_0]=2\l Y_0$. We can proceed and get in this way a non-trivial deformation,
but it will not be of the type described in Corollary \ref{corTan}.

So let us use the assumption on filtration, namely that after passing to grading we have a
monomorphism $\h\to\g$ (the filtration in $\g$ is $F_i=\oplus_{k\ge i}\g_k$). This
yields $\l=0$, and then it is not difficult to deduce that the brackets for $[Z_i,Z_j]$ are
the same as in $\g$.

Indeed, since $\op{ad}(S_1)[Z_i,Z_j]=(i+j-5)[Z_i,Z_j]$ the latter commutator can be non-zero
only for $i+j=4$ or $5$. Applying $\op{ad}(S_0)$ to $[Z_{i+1},Z_j]=0$ with $i+j=5$ we get
$[Z_i,Z_j]=-[Z_{i+1},Z_{j-1}]$, from where $[Z_i,Z_{5-i}]=(-1)^i2Y_0$.
Applying  $\op{ad}(S_0)$ to this starting from $i=0$ we get $[Z_i,Z_{4-i}]=0$.

Now by Leibniz rule it follows that $[S_i,Y_0]=0$, and so we  restore all commutators relations
from $\g$ for the subalgebra $\langle S_0,S_1,Y_0,Z_j\rangle$.

It remains to add $R$. By filtration reason the following relations should hold for some
constant coefficients:
 \begin{equation}\label{qwa}
[S_0,R]=\a Z_5+\b S_1+\gamma  R+\d Z_4,\ \ [S_1,R]=\z Z_5,\ \ [Z_5,R]=0.
 \end{equation}
Applying to the last equality $\op{ad}(S_0)$ several times (and using the first equality)
we successively get formulae for $[Z_i,R]$, which we do not reproduce all.
When $i=-1$ ($Z_{-1}=0$) we obtain
 \begin{multline*}
-3\gamma Z_0-\tfrac52\gamma(7\b-3\gamma)Z_1+5\gamma^2(7\b-\gamma)Z_2-
\tfrac32\gamma^3(21\b-5\gamma)\\
+\gamma^4(14\b-3\gamma)Z_4-\tfrac12\gamma^5(5\b-\gamma)Z_5-2(\a+\d\gamma)Y_0=0,
 \end{multline*}
which implies $\a=\gamma=0$.

Now applying $\op{ad}(S_0)$ to the middle equality in (\ref{qwa}) we obtain
$\b S_1+\frac52\delta Z_4=\z Z_4$, so that $\b=0$, $\z=\frac52\d$. The last number can be nonzero,
but it is removed by the change $R\mapsto R-\frac25\z Z_5$. This restores all the commutators
and so we get $h=\h\subset\g$. Thus no non-trivial deformation exists.
 \end{proof}

Thus we see that the rank 2 distribution in $\R^6$ corresponding to (\ref   {submax26})
actually represents a sub-maximal case. The structure of maximal symmetry
is unique according to \cite{AK,DZ}, but this is not so for the sub-maximal case;
here are some other models representing rank 2 distributions in $\R^6$ with
9-dimensional symmetry algebra:
 $$
y'=(z''')^2+\epsilon^2(z')^2\qquad \text{ and also }\qquad y'=(z''')^2+\epsilon^2z^2.
 $$

Another interesting series of symmetric models corresponding to rank 2 distributions in $\R^6$ is
 \begin{equation}\label{uv3m}
y'=(z''')^m.
 \end{equation}
For generic $m$ the symmetry algebra is 7-dimensional, but for some values like $\frac12$ and $\frac13$
it is 8-dimensional, and these special systems together with elliptic and hyperbolic Monge equations from
\cite{AK} yield sub-submaximal symmetric models of 2-distributions in $\R^6$.

In higher dimensions the maximally symmetric Monge equations are given by
 $$
y'=(z^{(n)})^2.
 $$
They represent rank 2 distributions in $\R^{n+3}$ with the symmetry algebra of dimension $2n+5$.
The sub-maximal case is realized by any of the equations
 $$
y'=(z^{(n)})^2+\epsilon(z^{(j)})^2
 $$
with $0\le j<n$; its symmetry algebra is $(2n+3)$-dimensional. Thus in this case the gap is 2.

For other equations, like elliptic and hyperbolic Monge equations \cite{AK}, the gap takes the minimal
possible value 1.

\section{Conclusion}\label{S5}

We have observed that the gap phenomenon exists (i.e. the lacune exceeds 1) if  the
maximal symmetry algebra has a semi-simple part; the larger is the dimension of the latter
the larger is expected the gap.

For geometric structures associated to distributions (in which case the maximal
symmetry algebra is graded and is given by the Tanaka theory) we have been
able to explain the phenomenon and presented a tool to perceive the gap.

Also in the latter case the maximal symmetry model (as well as its algebra)
is unique, though it is not true in general (e.g. for Riemannian metric with maximal
amount of Killing fields). Sub-symmetric models are seldom unique.

Remark also that maximal symmetry algebras act transitively on 
the manifold where the structure lives (in other words the maximal models are homogeneous).
This still holds for the sub-maximal models related to distributions,
e.g. in the discussed cases (\ref{submax25}) and (\ref{submax26}).
In general sub-maximal situation even this is not always true
(compare again the case of local Killing fields on surfaces).
In the lower-dimensional case like sub-submaximal (\ref{uv3m}), and also for Cartan's involutive systems with 6-dimensional algebras \cite[p.170]{C}, this property fails.

Understanding of maximal and sub-maximal symmetric models in general still remains an interesting open problem.


\end{document}